\documentclass{amsart}

\usepackage{amsmath}
\usepackage{amsfonts}
\usepackage{amssymb}

\date{}

\newcommand{\beqa}{\begin{eqnarray*}}
\newcommand{\eeqa}{\end{eqnarray*}}
\newcommand{\beqn}{\begin{eqnarray}}
\newcommand{\eeqn}{\end{eqnarray}}

\newcommand{\iy}{\infty}

\newcommand{\R}{\mathbb R}

\newcommand{\N}{\mathbb N}

\newcommand{\ov}{\overline}

\newcommand{\es}{\emptyset}

\newcommand{\f}{\frac}

\newcommand{\al}{\alpha}
\newcommand{\be}{\beta}

\newcommand{\de}{\delta}

\newcommand{\Om}{\Omega}

\newcommand{\Si}{\Sigma}

\newcounter{cnt1}
\newcounter{cnt2}
\newcounter{cnt3}
\newcommand{\blr}{\begin{list}{$($\roman{cnt1}$)$}
 {\usecounter{cnt1} \setlength{\topsep}{0pt}
 \setlength{\itemsep}{0pt}}}
\newcommand{\bla}{\begin{list}{$($\alph{cnt2}$)$}
 {\usecounter{cnt2} \setlength{\topsep}{0pt}
 \setlength{\itemsep}{0pt}}}
\newcommand{\bln}{\begin{list}{$($\arabic{cnt3}$)$}
 {\usecounter{cnt3} \setlength{\topsep}{0pt}
 \setlength{\itemsep}{0pt}}}
\newcommand{\el}{\end{list}}

\newtheorem{thm}{Theorem}[section]
\newtheorem{lem}[thm]{Lemma}
\newtheorem{cor}[thm]{Corollary}

\newtheorem{Def}[thm]{Definition}
\newtheorem{prop}[thm]{Proposition}
\newtheorem{rem}[thm]{Remark}
\newcommand{\Rem}{\begin{rem} \rm}
\newcommand{\bdfn}{\begin{Def} \rm}
\newcommand{\edfn}{\end{Def}}

\newcommand{\ba}{\begin{array}}
\newcommand{\ea}{\end{array}}

\usepackage{tikz}

\tikzstyle{vertex}=[scale=0.9,auto=left,circle,fill=black!10,inner
sep=0.9pt]
\usetikzlibrary {positioning}
\usetikzlibrary{arrows}
\usepackage{graphics}
\usepackage{graphicx}
\usepackage{hyperref}
\usepackage{graphicx,epsfig}
\usepackage{tikz}
\usepackage{caption}
\usepackage{subcaption}
\usepackage{float}
\usetikzlibrary{shapes,arrows}

\begin{document}
\sloppy

\title{Ordinal indices of Small Subspaces of $L_p$}

\author[S Dutta]{S Dutta}
\address[S Dutta]{Department of Mathematics and Statistics\\
Indian Institute of Technology Kanpur \\
India, \textit{E-mail~:} \textit{sudipta@iitk.ac.in}}

\author[D Khurana]{D Khurana}
\address[Divya Khurana]{Department of Mathematics and Statistics\\
Indian Institute of Technology Kanpur \\
India, \textit{E-mail~:} \textit{divyakh@iitk.ac.in}}



\begin{abstract}
We calculate ordinal $L_p$ index defined in \cite{BRS} for
Rosenthal's space $X_p$, $\ell_p$ and $\ell_2$. We show an
infinite dimensional subspace of $L_p$ $(2 < p < \infty)$
non isomorphic to $\ell_2$ embeds in $\ell_p$ if and only
if its ordinal index is minimum possible. We
also give a sufficient condition for a $\mathcal{L}_p$ subspace of
$\ell_p\oplus\ell_2$ to be isomorphic to $X_p$.
\end{abstract}

\maketitle

\section{Introduction}
Kadec and  Pelczynski in \cite{KP} proved that if $X$ is infinite
dimensional subspace of $L_p$ $(2 < p < \infty)$ then either $X$ is
isomorphic to $\ell_2$ or $X$ contains a isomorphic copy of
$\ell_p$. In addition if $X$ is complemented in $L_p \ (1<p<\iy)$
and $X$ is not isomorphic to $\ell_2$ then it contains a
complemented copy of $\ell_p$. They also proved that if $X$ is a
subspace of $L_p$ $(2 < p < \infty)$ such that $X$ is isomorphic to
$\ell_2$ then $X$ is complemented in $L_p$. In \cite{WO} it was
shown that if $X$ is a subspace of $L_p$ $(2 < p < \infty)$ such
that $\ell_2 \not\hookrightarrow X$ then $X\hookrightarrow\ell_p$.
Thus if $X$ is a subspace of $L_p$ $(\ 2 < p < \infty)$ such that
$X$ is not isomorphic to $\ell_2$ and $X\not\hookrightarrow \ell_p$
then $X$ contains an isomorph of $\ell_p\oplus\ell_2$. Moreover if
$X$ is $\mathcal{L}_p$ subspace not isomorphic to  $\ell_p$ then $X$
contains a complemented isomorph of $\ell_p\oplus\ell_2$. $\ell_p$,
$\ell_2$, $\ell_p\oplus\ell_2$ and $\ell_p(\ell_2)$  are referred as
small subspaces of $L_p$ $(2 < p < \infty)$ and for a long time
these were only known examples of complemented subspaces of $L_p$.
In 1970, Rosenthal \cite{Ro} constructed a complemented subspace of
$L_p$ $(1<p<\iy)$, which is denoted by $X_p$ and is not isomorphic
to the four spaces mentioned above. $X_p$ is also a kind of  small
subspace of $L_p$ in the sense that it embeds in $\ell_p \oplus
\ell_2$. In \cite{BRS} the authors defined an ordinal $L_p$ index
for separable Banach spaces. With help of this index they proved
that there are uncountably many mutually non isomorphic
$\mathcal{L}_p$ subspaces of $L_p$ $(1<p<\iy)$. But the exact value
of the $L_p$ index of the spaces constructed by them is not known so
far.

It was proved in \cite{WO1} that for $2<p<\iy$, if $X$ is a
$\mathcal{L}_p$ subspace of $\ell_p\oplus\ell_2$ with unconditional
basis then $X$ is isomorphic to one of the spaces $\ell_p$,
$\ell_p\oplus\ell_2$ or $X_p$. But for any general $\mathcal{L}_p$
subspace $X$ of $\ell_p\oplus\ell_2$ it is an open question whether
$X$ is isomorphic to one of the spaces $\ell_p$,
$\ell_p\oplus\ell_2$ or $X_p$ (see \cite{HOS}).

In this work we will first calculate the ordinal $L_p$ index defined in
\cite{BRS} for Rosenthal's space $X_p$, $\ell_p\oplus \ell_2$,
$\ell_p$ and $\ell_2$. We will show (Corollary~\ref{minimum}) that an
infinite dimensional subspace of $L_p$ $(2<p<\iy)$ non isomorphic to
$\ell_2$ embeds in $\ell_p$  if and only if its ordinal index
is minimum possible.

It follows that (see Theorem ~\ref{choice}) for any infinite
dimensional subspace of $L_p$ $(2<p<\iy)$, $L_p$ index can have the
following three possibilities $\omega_0$, $\omega_0\cdot 2$ or
greater than equal to $\omega_0^2$.

Coming back to the question mentioned above, in Theorem~\ref{Xp} we
will provide a sufficient condition (which is trivially necessary)
for a $\mathcal{L}_p$ $(2<p<\iy)$ subspace of $\ell_p\oplus\ell_2$
to be isomorphic to $X_p$.

We now provide basic background for our work. For notation we
closely follow \cite{DA} and \cite{BRS}.

\textbf{Notation}: Let $X$ and $Y$ be real Banach spaces. By
$X\stackrel{(c)}\hookrightarrow Y$, we mean $X$ is isomorphic to a
complemented subspace of $Y$. $X\equiv Y$ means $X$ is isometric to
$Y$. By $X\stackrel{k}\sim Y$, we mean $X$ is isomorphic to $Y$ and
there is an isomorphism $S:X\rightarrow Y$ such that
$||S||||S^{-1}||\leq k$. We will write $X\stackrel{k}\hookrightarrow
Y$ if  $X\stackrel{k}\sim Z$ for some isometric subspace $Z$ of
$Y$.\\
If $X$ is a subspace of $L_p$ by $X_0$ we denote the subspace of $X$
consisting of mean zero functions only.\\
For $2<p<\iy$, we will denote by $R_p$ the constant of equivalence in
Rosenthal's inequality \cite[Theorem 3]{Ro}. If in the context $p$
is fixed, we will simply denote it by $R$.

We now recall the notion of independent sum  and  $R_p^{\al}$ spaces.

Consider the sequence $\{X_n\}$ of Banach spaces where $X_n$
is subspace of $L_p(\Omega_n, \mu_n)$ for some probability measure $\mu_n$.
Let  $\mu=\Pi \mu_n$ be the product measure on $\Omega= \Pi \Omega_n$.
For each $n\in\N$, we denote the canonical projection  from
$\Omega$ to $\Omega_n$ by $P_n$ and $j_n(f)=foP_n$. Let $X$ be the subspace of
$L_p(\Omega, \mu)$ consisting of constant functions and $j$ be the
inclusion of $X$ into $L_p(\Omega, \mu)$. By $(\Si X_n)_{Ind,p}$
we denote the closed linear span of
$\cup j_n(X_n) \cup j(X)$ in $L_p(\Omega, \mu)$.\\

In defining $R_p^{\al}$ spaces we follow the view point considered in
\cite{DA}.\\
For $1\leq p< \iy$, let
$R^{0}_p=L_p^0=[1]_{L_p}$. Now suppose that $R^{n}_p$ has been defined for
$n\in \N$. We define $R^{n+1}_p=\ell_p^2\otimes R^{n}_p$ and
$R^{\omega_0}_p=(\Si R^{n}_p)_{Ind,p}$. Thus we have
$R^{\omega_0}_p=(\Si \ell_p^{2^n})_{Ind,p}$.  In general for any
ordinal $\al<\omega_1$, we define $R_p^{\al+1}=\ell_p^2\otimes R_p^{\al}$ and
for a limit ordinal $\be$ we put $R_p^{\be}=(\Si_{\al<\be} R_p^{\al})_{Ind,p}$.\\
It is known that Rosenthal's space $X_p$ is isomorphic to
$R^{\omega_0}_p$ for  $1 < p < \infty$.

Now we will define the notion of $(p,2,(w_n))$ sum of subspaces of
$L_p$ given in \cite{DA}.

\bdfn Let $(X_n)$ be a sequence of subspaces of $L_p(\Om,\mu)$
for some probability measure $\mu$ and
$\{w_n\}$ be a sequence of real numbers, $0\leq w_n \leq 1$. For any
sequence $(x_n)$ such that $x_n\in X_n$, let
 \begin{align*}
 ||(x_n)||_{p,2,(w_n)}=max\{(\sum ||x_n||_p^p)^{1/p}, (\sum
 ||x_n||_2^2w_n^2
 )^{1/2}\}
 \end{align*}
 and
 \begin{align*}
 X=(\Si X_n)_{p,2,(w_n)}=\{(x_n):x_n\in X_n~ for ~all ~n ~and
 ~||(x_n)||_{p,2,(w_n)}< \iy\}.
\end{align*}
 \edfn

\begin{rem}\label{sum}
If $w_n=1$ for all $n\in\N$ then we will denote $(\Si
X_n)_{p,2,(w_n)}$ by $(\Si X_n)_{p,2,(1)}$. From above definition we
can easily verify that if $X$ is a subspace of $L_p$ then $(\sum
X)_{p,2,(1)}$ is stable under taking $(p,2,(1))$ sum. We will denote
$(\sum X)_{p,2,(1)}$ by $(X)_{p,2,(1)}$. Thus $X \sim (X)_{p,2,(1)}$
if and only if $X$ is stable under taking $(p,2,(1))$ sum.
\end{rem}

\begin{rem}\label{stable}
In \cite [Lemma 2.1]{DA} it was proved that if $X_n$ is
a subspace of $L_p(\Omega_n,\mu_n)$ for each $n\in\N$
 then $(\Si X_n)_{Ind,p}$ is isomorphic to
$(\Si X_{n,0})_{Ind,p}\oplus L_p^0$. Thus it follows from
Rosenthal's inequality and Remark ~\ref{sum} that for each limit ordinal
$\omega_0\leq \al <\omega_1$, $R_p^{\al}$ $(2<p<\iy)$ is stable under taking
$(p,2,(1))$ sum. It is known that
(see \cite [Corollary 2.10]{DA}) $R_p^{\al}$ $(2<p<\iy)$ spaces are
isomorphically distinct at limit ordinals. Thus for each ordinal
$\omega_0\leq \al <\omega_1$, $R_p^{\al}$ $(2<p<\iy)$ is stable under taking
$(p,2,(1))$ sum.
\end{rem}

We will use the following result, proof of which is
essentially contained in the proof of
\cite[Theorem 2.4]{DA}.

\begin{thm}\label{alspach}
Let $2<p<\iy$. There exists a constant $A$ such that
$R_p^{\omega_0\cdot2}\stackrel{A} \sim (\sum
R_{p,0}^{\omega_0+n})_{p,2,(1)}$.
\end{thm}

We now recall the definition of ordinal $L_p$ index for separable Banach
spaces from \cite{BRS}. \bdfn For $ n \in \N \cup \{0\}$, let $D_n$
be the set of all $n$-strings of $0's$ and $1's$. For a separable
Banach space $X$, $X^{D_n}$ be the set of all functions from $D_n$
to $X$. Let $X^\mathcal{D} = \cup_{n=0}^\infty X^{D_n}$.

For $u \in X^\mathcal{D}$ we write $|u|=n$ if  $ u \in X^{D_n}$. For
$t\in D_n$ and $s\in D_m$ we denote by $t\cdot s$ element in
$D_{n+m}$  given by $t_1\cdots t_{n}\cdot s_1\cdots s_{m}$.

On $X^\mathcal{D}$ we define a strict partial order by $ u \prec v$
if $|u| < |v|$ and for $ k = |v| - |u|$, $u(t) = 2^{-{k/p}}
\sum\limits_{s \in {D_k}} v(t\cdot s)$.

Let $1\leq p<\infty$ and $0< \de \leq 1$. Let ${\ov X}^\de$ be the
set of all $u \in X^\mathcal{D}$ such that
\begin{equation} \label{eq1}
\de (\sum\limits_{t \in D_{|u|}} |c(t)|^p)^{1/p} \leq \parallel
\sum\limits_{t \in D_{|u|}} c(t) u(t)\parallel _X \leq
(\sum\limits_{t \in D_ {|u|}}|c(t)|^p)^{1/p}
\end{equation}
for all $ c \in {\R}^{D_{|u|}}$.

\begin{rem} \label{independent}
As a consequence of (\ref{eq1}) we observe that if
$u=(u_1,\cdots,u_{2^k})\in {\ov X}^\de$ then
$\{u_1,\cdots,u_{2^k}\}$ are linearly independent. In case of
$\de=1$, $\{u_1,\cdots,u_{2^k}\}$ spans $\ell_p^{2^k}$
isometrically.
\end{rem}

Let $H_0^{\de}(X) = \ov X^{\de}$. If $\al = \be + 1$ and
$H_{\be}^{\de}(X)$ has been defined, then take
\begin{align*}
 H_{\al}^{\de}(X) = \{ u \in H_{\be}^{\de}(X) : u \prec v~\mbox
 {for some}~v \in H_{\be}^{\de}(X)\}.
\end{align*}
If $\al$ is a limit ordinal we define $H_{\al}^{\de}(X) = \cap_{\be
< \al} H_{\be}^{\de}(X)$.

In \cite {BRS} it was proved that for $1\leq p <\iy$
if $L_p\not\hookrightarrow X$ then for any
$0<\de\leq 1$ there exists an
ordinal $\be<\omega_1$ such that
$H_{\be}^{\de}(X)=H_{\be+1}^{\de}(X)$. Let $ h_p(\de, X)$ be the
least ordinal $\al$ such that $H_{\al}^{\de}(X) =
H_{\al+1}^{\de}(X)$. If $L_p\not \hookrightarrow X$, we define $
h_p(X) = sup_{0<\de \leq 1} h_p(\de,X)$. If $L_p\hookrightarrow X$
by convention we take $h_p(X) = \omega_1$. In \cite{BRS} it was
proved that for $1\leq p <\iy$, $h_p(X)<\omega_1$ if and only if
$L_p\not\hookrightarrow X$.  \edfn

We will be using the following results repeatedly while calculating
the $L_p$ index of $R_p^{\omega_0}$, $\ell_p$ and $\ell_2$.

\begin{thm}\cite[Theorem 2.1]{BRS} \label{subspace}
If $X$ and $Y$ are  two separable Banach spaces such that $X\hookrightarrow Y$
then $h_p(X)\leq h_p(Y)$.
\end{thm}

\begin{thm} \cite[Theorem 2.4]{BRS} \label{h1}
Let $1\leq p<\iy$, $0\leq \al <\omega_1$. Then $1\in
H_{\al}^1(R^{\al}_p)$.
\end{thm}

\textbf{FACT} \cite [Lemma 2.5]{BRS}: Let $X$ be a separable Banach,
$0<\de\leq 1$ and $\al<\omega_1$. Let $e\in H_{\al}^{\de}(X)$. Let
$\ov{e}$ be the element of $(X\oplus X)_p^\mathcal{D}$ defined by
$\ov{e}(t)=2^{\f{-1}{p}}(e(t)\oplus e(t))$ for all $t\in D_{|e|}$.
Then $\ov{e}\in H_{\al+1}^{\de}((X\oplus X)_p)$. As a consequence if
$X$ is isomorphic to its square and $h_p(X)>\al$, for some limit
ordinal $\al$ then $h_p(X)\geq\al+\omega_0$.

\begin{rem}\label{infinite}
Let $1\leq p<\iy$. It is easy to observe that
$h_p(\ell_p^{2^n})=n+1$. Hence by Theorem~\ref{subspace} for any
infinite dimensional space $X$, we have $h_p(X)\geq \omega_0$.
\end{rem}

\begin{rem}\label{r1}
$(a)$ It was proved by Rosenthal in \cite{Ro} that for
$1<p<\iy$, $\ell_p(\ell_2)\not\hookrightarrow X_p$ (which is
isomorphic to $R_p^{\omega_0}$).\\
$(b)$ For a sequence of scalars $\{w_n\}$, $0\leq w_n \leq 1$
and $2<p<\iy$ we recall the Rosenthal's condition which is:
for each $\epsilon>0$, $\Si_{w_n<\epsilon}{w_n}^{\f {2p}{p-2}}=\iy.$\\
If the sequence $\{w_n\}$ satisfies Rosenthal's condition then it was
proved in \cite[Theorem 2.4]{DA} that $R_p^{\omega_0\cdot2}\sim
(R_p^{\omega_0})_{p,2,(w_n)}$. It follows from \cite[Proposition
2.11]{DA}  that by breaking the collection $\{w_n\}$ into three
disjoint sub-collections namely $\{w_n^1\}$, $\{w_n^2\}$ and
$\{w_n^3\}$ satisfying the conditions; $\{w_n^1\}$ satisfies
Rosenthal's condition, $inf w_n^2>0$ and $\Si{(w_n^3)}^{\f
{2p}{p-2}}<\iy$, we get
\begin{equation*}
R_p^{\omega_0\cdot2}\sim (R_p^{\omega_0})_{p,2,(w_n^1)}\oplus
(R_p^{\omega_0})_{p,2,(1)} \oplus \ell_p(R_p^{\omega_0}).
\end{equation*}
 Since $\ell_2 \hookrightarrow R_p^{\omega_0}$, we get
$\ell_p(\ell_2)\hookrightarrow R_p^{\omega_0\cdot2}$.

$(c)$ It follows from $(a)$ and $(b)$ above that
$R_p^{\omega_0\cdot2}\not \hookrightarrow R_p^{\omega_0}$.

\end{rem}

\section{Main Results}

The following Lemma is  key to calculate $L_p$ index of $\ell_p$,
$\ell_2$ and $R_p^{\omega_0}$.

\begin{lem}\label{p,2,1}
Let $2<p<\iy$ and $X$ be a subspace of $L_p$. If for some
$0<\de\leq1$ and $n\in \N$, $H^{\de}_{\omega_0+n}(X)\not=\es$ then
there exists a constant $C$ (depending on $\de$, $p$ and $X$ only) such
that $R_{p,0}^{\omega_0+n}\stackrel{C}\hookrightarrow
(X)_{p,2,(1)}$.
\end{lem}
\begin{proof}
Without loss of generality we assume that $X$ consists of only mean
zero functions (otherwise we write $X= X_0\oplus L_p^0$ and work
with $X_0$). Let $u_n \in H_{\omega_0+n}^{\de}(X)$, $v_n \in
H_{\omega_0}^{\de}(X)$ such that $u_n\prec v_n$. So we have
$|v_n|\geq n$. For all $k\in \N$ we can find some $v_k^n\in
H_{k}^{\de}(X)$ such that $v_n\prec v_k^n$. Further for all $k$ we
can find some $w_k^n\in H_{0}^{\de}(X)$ such that $v_k^n\prec w_k^n$ and
$|w_k^n|\geq n+k+1$.

Let $|v_n|=m$. Then $m\geq n$ and $|w_k^n|\geq m+k+1$. For a fixed string
$t_1\cdot\cdot\cdot t_{m}$ of $0's$ and $1's$ we have
 \begin{equation} \label{eq2}
v_n(t_1 \cdots t_{m})=2^{\f{-({|w_k^n|-m})}{p}} \sum\limits_{s\in
D_{|w_k^n|-m}}w_k^n(t_1\cdot\cdot\cdot t_{m}\cdot s).
\end{equation}
Let $W_k^n({t_1\cdots t_{m}})$ be the subspace spanned by components
of $w_k^n$ which appear in the representation of
$v_n(t_1\cdots t_{m})$ in (\ref{eq2}) above. It is immediate
to observe that $\ell_p^{2^k}\stackrel{\f{1}{\de}}\hookrightarrow
W_k^n({t_1\cdots t_{m}})$ and this copy consists of mean zero
functions only (see Remark~\ref{independent}). Let $X_{t_1\cdots
t_{m}}=(W_k^n({t_1\cdots t_{m}}))_{Ind,p}$ , then
$(\ell_p^{2^k})_{Ind,p}\stackrel{\f{1}{\de}}\hookrightarrow
X_{t_1\cdots t_{m}} $, that is
\begin{equation}\label{eq3}
R_{p,0}^{\omega_0}\stackrel{\f{1}{\de}}\hookrightarrow X_{t_1\cdots
t_{m}}.
\end{equation}
By taking $(p,2,(1))$ sum on both sides of (\ref{eq3}) we
have $(\sum\limits_1^{2^m}R_{p,0}^{\omega_0})_{p,2,(1)}
\stackrel{\f{1}{\de}}\hookrightarrow ((X)_{Ind,p})_{p,2,(1)}$.
It follows from Rosenthal's inequality that
$(X)_{Ind,p}\stackrel{R}\sim(X)_{p,2,(1)}$ and
$(\sum\limits_1^{2^m}R_{p,0}^{\omega_0})_{Ind,p}
\stackrel{R}\sim(\sum\limits_1^{2^m}R_{p,0}^{\omega_0})_{p,2,(1)}$.
Since $R_{p,0}^{\omega_0+m}\equiv
(\sum\limits_1^{2^m}R_{p,0}^{\omega_0})_{Ind,p}$ so we have
$R_{p,0}^{\omega_0+m}\stackrel{\f {R^2}{\de}}\hookrightarrow
{((X)_{p,2,(1)})}_{p,2,(1)}$. Thus we have
$R_{p,0}^{\omega_0+m}\stackrel{\f {R^2B}{\de}}\hookrightarrow
(X)_{p,2,(1)}$ where $B$ is the constant such that
$(X)_{p,2,(1)}\stackrel{B}\sim((X)_{p,2,(1)})_{p,2,(1)}$. Since
$m\geq n$ so we get the desired result.
\end{proof}

We will use the following result for calculating the $L_p$ index of
$\ell_p$ and $\ell_2$ which we will prove later.
\begin{thm}\label{index}
Let $2<p<\iy$. Then $h_p(R_p^{\omega_0})=\omega_0\cdot2$.
\end{thm}

We believe the following Lemma is essentially
known, however we include the proof for completion.
Let $\{e_n\}$ denotes the standard unit vector basis of $\ell_p$.

\begin{lem}\label{l4}
Let $1 \leq p < \infty, \ u=(u_1,\cdots, u_{2^k})\in H_0^1(\ell_p)$
and $u_i=\sum\limits_{j\in {N_i} }{b_j^i e_j}$, where $N_i\subseteq
\N$, $1\leq i\leq 2^k$. Then for $i\not=j$, $1\leq i,j \leq 2^k$
 we have $N_{i}\cap N_{j} = \es$.
\end{lem}
\begin{proof}
If possible let $N_1\cap N_2\not=\es$  and ${n_0}\in N_1\cap N_2$.
Since $u\in H_0^1(\ell_p)$, we have $||
c_1u_1+c_2u_2||_p^p=|c_1|^p+|c_2|^p$,
$\sum\limits_{j\in N_i} {|b_j^i|^p}=1$ for all $c_1,c_2\in\R$ and $1\leq i \leq 2^k$.
Then we have
\begin{eqnarray*}
||c_1u_1+c_2u_2||_p^p &\leq& |c_1b_{n_0}^1 +c_2b_{n_0}^2|^p+
|c_1|^p{\sum\limits_{i\neq {n_0}}} {|b_i^1|^p}
+|c_2|^p{\sum\limits_{i\neq {n_0}}}{|b_i^2|^p}.
\end{eqnarray*}
Taking $c_1=b_{n_0}^2$ and $c_2=-b_{n_0}^1$, we get
\begin{eqnarray*}
||c_1u_1+c_2u_2||_p^p &\leq & |c_1|^p{\sum\limits_{i\neq {n_0}}}
{|b_i^1|^p}+|c_2|^p{\sum\limits_{i\neq {n_0}}}{|b_i^2|^p}\\
&=&|c_1|^p(1-|b_{n_0}^1|^p)+|c_2|^p(1-|b_{n_0}^2|^p)\\
&<&|c_1|^p+|c_2|^p.
\end{eqnarray*}

This contradicts that $||c_1u_1+c_2u_2||_p^p=|c_1|^p+|c_2|^p$.

Hence $N_1\cap N_2=\es$. Similarly we can show that $N_i\cap
N_j=\es$ for $i\not=j$, $1\leq i,j\leq 2^k$.
\end{proof}

\begin{lem} \label{l5}
Let $1\leq p <\iy$. Then $h_p(1, \ell_p)=\omega_0$.
\end{lem}

\begin{proof}
Let $u=(u_1,\cdots, u_{2^k})\in H_0^1(\ell_p)$ and
$u_{i}=\sum\limits_{j\in {N_{i}}}{a_j^{i}e_j}$, where $N_{i}$  is a
subset of $\N$, $1\leq i\leq 2^k$. If there doesn't exists any
$N_u<\omega_0$ such that $u\not\in H_{N_u}^1(\ell_p)$, we can find
$v_n\in H_0^1(\ell_p)$ with $|v_n|\uparrow \omega_0$ and $u\prec
v_n$ for each $n$. Thus for each $i$ we can find some fixed $k$
string $t_1\cdots t_k$ of $0's$ and $1's$ such that
$u_{i}=2^{\f{-(|v_n|-k)}{p}}\sum\limits_{s\in
{D_{|v_n|-k}}}v_n(t_1\cdot\cdot\cdot t_{k}\cdot s )$. Using
Lemma~\ref{l4} we have for $s\in{D_{|v_n|-k}}$ and any $k$ string
$t_1\cdot\cdot\cdot t_{k}$ of $0's$ and $1's$,
$v_n(t_1\cdot\cdot\cdot t_{k}\cdot s)=\sum\limits_{j\in N_s}
c_j^ne_j$, where $\cup_{s\in D_{|v_n|-k}} N_s=N_{i}$ and
$N_{s_1}\cap N_{s_2}=\es$ if $s_1\neq s_2$. Thus for each $j\in N_{i}$
and $n\in\N$ can find some $j_0$ such that
$a_j^{i}={\f{c_{j_0}^n}{\f{2^{{|v_n|-k}}}{p}}}$,
.

Since for all $n$ and $j\in N_{i}$ we have $|c^n_{j}| \leq 1$ we
conclude that $a_j^{i}=0$ for all $j\in \N_{i}$, which contradicts
$||u_{i}||=1$, $1\leq i\leq 2^k$. Thus there exists some
$N_u<\omega_0$ such that $u\not\in H_{N_u}^1(\ell_p)$. Hence
$h_p(1,\ell_p)\leq \omega_0$. But it is easy to observe that $h_p(1,
\ell_p^{2^n})=n+1$. Thus we have $h_p(1,\ell_p)=\omega_0$.
\end{proof}

\begin{prop} \label{ellp1}
Let $2<p<\iy$. Then $h_p(\ell_p)=\omega_0.$
\end{prop}
\begin{proof}
It follows from Lemma~\ref{l5} that $h_p(\ell_p)\geq\omega_0$. If
$h_p(\ell_p)>\omega_0$ then by using FACT we have
$h_p(\ell_p)\geq\omega_0\cdot2.$ We know that $\ell_p\hookrightarrow
R_p^{\omega_0}$ thus  Theorem ~\ref{index} and Theorem
~\ref{subspace} implies that $h_p(\ell_p)=\omega_0\cdot 2.$ Thus we can find some
$0<\de\leq1$ such that $H^{\de}_{\omega_0}(\ell_p)\not=\es$.
Then again by FACT we have $H^{\de}_{\omega_0+1}((\ell_p\oplus \ell_p)_p)
\not=\es$. Since $(\ell_p \oplus \ell_p)_p \equiv \ell_p$, we have
$H^{\de}_{\omega_0+1}(\ell_p)\not=\es$. By similar arguments we can show that
$H^{\de}_{\omega_0+n}(\ell_p)\not=\es$ for all $n\in \N$. Thus by
Lemma~\ref{p,2,1} there exists some constant $C$ such that
 $R_{p,0}^{\omega_0+n}\stackrel{C}\hookrightarrow
(\ell_p)_{p,2,(1)}$ for all $n\in\N$ . By taking $(p,2,(1))$
sum on both sides we have
$(\sum R_{p,0}^{\omega_0+n})_{p,2,(1)}\stackrel{C}\hookrightarrow
({(\ell_p)_{p,2,(1)}})_{p,2,(1)}$. Now it follows from Theorem~\ref{alspach} that
$R_{p}^{\omega_0\cdot2}\stackrel{CA}\hookrightarrow
((\ell_p)_{p,2,(1)})_{p,2,(1)}.$ But again by Remark~\ref{stable}
there exists some constant $B$ such that
$(R_p^{\omega_0})_{p,2,(1)}\stackrel{B} \sim R_p^{\omega_0}$. Also
we have $((\ell_p)_{p,2,(1)})_{p,2,(1)}\stackrel{(c)}\hookrightarrow
((R_p^{\omega_0})_{p,2,(1)})_{p,2,(1)}\stackrel{B^2}\sim
R_p^{\omega_0}$. Thus
$R_{p}^{\omega_0\cdot2}\stackrel{CAB^2}\hookrightarrow
R_p^{\omega_0}$. Which is a contradiction to Remark~\ref{r1}. Thus
$h_p(\ell_p)=\omega_0.$
\end{proof}

Now we will calculate the ordinal $L_p$ index of $\ell_2$.
\begin{prop}\label{ell2}
Let $2<p<\iy$. Then $h_p(\ell_2)=\omega_0.$
\end{prop}
\begin{proof}
It follows from  Remark~\ref{infinite} that $h_p(\ell_2)\geq \omega_0$. If we
assume $h_p(\ell_2)>\omega_0$ then by FACT it follows that $h_p(\ell_2) \geq
\omega_0\cdot2.$ Since $\ell_2\hookrightarrow R_p^{\omega_0}$, by
Theorem~\ref{index} and Theorem~\ref{subspace} we have that
$h_p(\ell_2)=\omega_0\cdot 2$. Thus we can find some $0<\de \leq 1$ such
that $H^{\de}_{\omega_0+1}(\ell_2)\not=\es$. By using
Lemma~\ref{p,2,1} we have $R_{p,0}^{\omega_0}\hookrightarrow
(\ell_2)_{p,2,(1)}$. Remark ~\ref{stable} and the fact
 $\ell_2 \sim (\mathbb{R})_{p,2,(1)}$ implies that
  $\ell_2$ is stable under taking $(p,2,(1))$ sum. Thus we conclude
$R_{p}^{\omega_0}\hookrightarrow \ell_2$, which is not possible as
 $\ell_p\hookrightarrow R_p^{\omega_0}$. This shows $h_p(\ell_2)=\omega_0.$
\end{proof}

For $2<p<\iy$, let $X$ be infinite dimensional subspace of $L_p$.
In the following Theorem we show that value of $L_p$ index of $X$
can have only three possible choices.

\begin{thm} \label{choice}
Let $2 < p < \infty$ and $X$ be an infinite dimensional subspace
of $L_p $. Then either $h_p(X)=\omega_0$, $h_p(X)=\omega_0\cdot2$
or $h_p(X)\geq\omega_0^2$.
\end{thm}
\begin{proof}
We recall from \cite{HOS} that if $X$ is a subspace of $L_p$ then
either $X\hookrightarrow \ell_p\oplus \ell_2$ or
$\ell_p(\ell_2)\hookrightarrow X$.
Suppose $X\hookrightarrow \ell_p\oplus \ell_2$. Note that $R_p^{\omega_0}$
and $\ell_p\oplus\ell_2$ embeds in each other (see \cite{Ro}).
Thus we have from Theorem ~\ref{index} and Theorem ~\ref{subspace}
that $h_p(X)\leq \omega_0\cdot 2$. If $X\hookrightarrow \ell_p$ or $X\sim \ell_2$
then by Proposition ~\ref{ellp1} and
Proposition ~\ref{ell2} we have $h_p(X)=\omega_0$. In otherwise
$\ell_p\oplus\ell_2\hookrightarrow X$ hence $h_p(X)=\omega_0\cdot2$.\\
In the case when $\ell_p(\ell_2) \hookrightarrow X $, we actually have
$R_p^{\omega_0\cdot n}\hookrightarrow X $ for each $n\in\N$. This is because
$R_p^{\omega_0\cdot n}\hookrightarrow\ell_p(\ell_2)$.
(This fact is highly non trivial.
In \cite{S} Schechtman constructed countably many mutually
non isomorphic subspaces of $\ell_p(\ell_2)$, by taking repeated tensor
product of $X_p$, denoted by $\otimes_1^n X_p$. These spaces are
$\mathcal{L}_p$ spaces. Alspach proved in \cite{DA} that
$R_p^{\omega_0\cdot n}\hookrightarrow \otimes_1^n X_p$.)

It was proved in \cite{BRS} that for any ordinal
$\al<\omega_1$, $h_p(R_p^{\al})\geq \al+1$.
Hence we have $h_p(X)\geq\omega_0\cdot n +1$ for all $n\in\N$.
This shows that $h_p(X)\geq \omega_0^2$.

\end{proof}

We now prove Theorem ~\ref{index}.\\
\textbf{Proof of Theorem~\ref{index}}: It is known that
(see \cite [Corollary 2.10]{DA}) $R_p^{\al}$ spaces are
isomorphically distinct at limit ordinals. Thus
we have $R_p^{\omega_0+k}\sim R_p^{\omega_0}$ for all $k\in \N$ and
using Theorem ~\ref{h1} we have
$h_p(R_p^{\omega_0})\geq \omega_0\cdot 2$. We will show that
$h_p(R_p^{\omega_0})=\omega_0\cdot 2$. Suppose on the contrary
$h_p(R_p^{\omega_0})>\omega_0\cdot 2$. Hence there exists a
$0<\de\leq 1$ such that $h_p(\de, R_{p}^{\omega_0})>\omega_0\cdot
2$. Thus for each $n\in\N$,
$H_{\omega_0+n}^{\de}(R_{p}^{\omega_0})\not = \es$. By using
Lemma~\ref{p,2,1} there exists some constant $C$ such that for all
$n\in\N$  we have $R_{p,0}^{\omega_0+n}\stackrel{C}\hookrightarrow
(R_{p}^{\omega_0})_{p,2,(1)}$. By Taking (p,2,(1)) sum we have
$(\sum R_{p,0}^{\omega_0+n})_{p,2,(1)}\stackrel{C}\hookrightarrow
((R_{p}^{\omega_0})_{p,2,(1)})_{p,2,(1)}$. Now by using
Theorem~\ref{alspach} and Remark~\ref{stable}
we have  $R_{p}^{\omega_0\cdot
2}\stackrel{A}\sim (\sum R_{p,0}^{\omega_0+n})_{p,2,(1)}$ and
$(R_{p}^{\omega_0})_{p,2,(1)}\stackrel{B}\sim R_{p}^{\omega_0} $ for
some constant $B$. From this we conclude that
$R_{p}^{\omega_0\cdot2}\stackrel {CAB^2}\hookrightarrow
R_{p}^{\omega_0}$, which contradicts Remark~\ref{r1}. Thus we have
$h_p(R_p^{\omega_0})=\omega_0\cdot 2$.

\begin{cor} \label{minimum}
Let $X$ be an infinite dimensional subspace of $L_p$ $(2<p<\iy)$
such that $X\not\sim\ell_2$. Then $h_p(X)=\omega_0$ if and only if
$X\hookrightarrow \ell_p$.
\end{cor}
\begin{proof}

If $X\hookrightarrow\ell_p$ then by Proposition~\ref{ellp1},
Theorem~\ref{subspace} and Remark~\ref{infinite} we have
$h_p(X)=\omega_0$. Conversely if $X\not\sim\ell_2$ and
$X\not\hookrightarrow\ell_p$ then $\ell_p\oplus \ell_2
\hookrightarrow X$. Since $\ell_p\oplus \ell_2$ and $R_p^{\omega_0}$
embeds in each other
 thus by Theorem~\ref{subspace} and Theorem~\ref{index} we have
$h_p(X)\geq\omega_0 \cdot 2$.
\end{proof}

To end this note we will provide a sufficient condition (which is
trivially necessary) for a $\mathcal{L}_p$ subspace of
$\ell_p\oplus\ell_2$ to be isomorphic to $X_p$.

\begin{thm}\label{Xp}
Let $X$ be a $\mathcal{L}_p$ subspace of
$\ell_p\oplus\ell_2$, $ 2 < p < \infty$ such that
$X\not\hookrightarrow \ell_p$. Then $X\sim X_p$ if
and only if $X$ is stable under taking $(p,2,(1))$
sum.
\end{thm}
\begin{proof}
If $X$ is a $\mathcal{L}_p$ subspace of $\ell_p\oplus\ell_2$ such
that $X\not\hookrightarrow \ell_p$ then
$\ell_p\oplus\ell_2\stackrel{(c)}\hookrightarrow X$. By taking
$(p,2,(1))$ sum on both sides we get
$(\ell_p\oplus\ell_2)_{p,2,(1)}\hookrightarrow (X)_{p,2,(1)}$ and by
\cite[Lemma 2.5]{DA} this copy is complemented. Thus
$(\ell_p\oplus\ell_2)_{p,2,(1)}\stackrel{(c)}\hookrightarrow
(X)_{p,2,(1)}$. We claim that $(\ell_p\oplus\ell_2)_{p,2,(1)}\sim
R_p^{\omega_0}$. To see this first note that
$\ell_p^{2^n}\stackrel{(c)}\hookrightarrow \ell_p$ and for each $n$
the projection constant is 1. Thus we have $(\sum
\ell_p^{2^n})_{p,2,(1)}\stackrel{(c)}\hookrightarrow
(\ell_p)_{p,2,(1)}$. Since $R_p^{\omega_0} \sim (\sum
\ell_p^{2^n})_{p,2,(1)}$ there exists an isomorphic copy of
$R_p^{\omega_0}$ complemented in $(\ell_p)_{p,2,(1)}$. Observe that
this copy is stable under taking $(p,2,(1))$ sum.
 Since $\ell_p$ is complemented in this copy we have
$(\ell_p)_{p,2,(1)}$ is complemented there. As both concerned spaces
are isomorphic to their square  by decomposition method we have
$(\ell_p)_{p,2,(1)}\sim R_p^{\omega_0}$. Now $(\ell_p\oplus
\ell_2)_{p,2,(1)}\sim (\ell_p)_{p,2,(1)}\oplus
(\ell_2)_{p,2,(1)}\sim R_p^{\omega_0}\oplus \ell_2\sim
R_p^{\omega_0}$. Coming back to the proof, by the claim we have $X$
contains a complemented copy of $R_p^{\omega_0}$. Since $X$ is a
$\mathcal{L}_p$ subspace of $\ell_p\oplus\ell_2$ so by \cite
[Proposition 8.6]{HOS} we have $X\stackrel{(c)}\hookrightarrow X_p$.
Recall that $R_p^{\omega_0}\sim X_p$. Since $X$ is stable under
taking $(p,2,(1))$ sum it is isomorphic to its square. Using
decomposition method we have $X \sim X_p$.
\end{proof}

\begin{rem}
It follows a complemented subspace $X$ of $X_p$ is stable under
taking $(p,2,(1))$ sum if and only if $X$ satisfies condition $(2)$
of \cite[Theorem 2.1]{DA1}.
\end{rem}

\end{document}